\documentclass{amsart}

\newtheorem{theorem}{Theorem}[section]
\newtheorem{lemma}[theorem]{Lemma}

\theoremstyle{definition}
\newtheorem{definition}[theorem]{Definition}
\newtheorem{example}[theorem]{Example}

\theoremstyle{remark}
\newtheorem{remark}[theorem]{Remark}

\theoremstyle{coro}
\newtheorem{coro}[theorem]{Corollary}

\newtheorem{prop}[theorem]{Proposition}

\numberwithin{equation}{section}

\input{xy}
\xyoption{all}
\xyoption{poly}
\usepackage[all]{xy}
\usepackage{latexsym}
\usepackage{graphicx}
\usepackage{float}  

\usepackage[
  breaklinks,
  naturalnames,
  ]{hyperref}
  
\let\oldqed\qed
\renewcommand\qed{\oldqed\par\bigskip}

\newcommand\ZZ{{\mathbb{Z}}}
\newcommand\NN{{\mathbb{N}}}

\def\Ker{{\mathrm {Ker}}}

\def\Hom{{\mathrm {Hom}}}

\def\Tor{{\mathrm {Tor}}}

\def\gldim{{\mathrm {gldim}}}
\def\hhdim{{\mathrm {hhdim}}}

\begin{document}

\title{Two classes of algebras with infinite Hochschild homology }

\author{Andrea Solotar}
\address{Departamento de Matem\'atica, Facultad de Ciencias Exactas y Naturales,
 Universidad de Buenos Aires, Ciudad Universitaria, Pabell\'on 1
1428, Buenos Aires, Argentina}
\email{asolotar@dm.uba.ar}

\author{Micheline Vigu\'e-Poirrier}
\address{Laboratoire Analyse, G\'eom\'etrie \& Applications, UMR CNRS 7539,
Institut Galil\'ee, Universit\'{e}  Paris 13
F-93430 Villetaneuse, France}
\email{vigue@math.univ-paris13.fr}
\thanks{This work has been supported by the projects UBACYTX212 and PIP-CONICET 5099.
The first author is a research member of CONICET (Argentina) and a
Regular Associate of ICTP Associate Scheme. The second author is a research member of University of Paris 13,
 CNRS, UMR 7539 (LAGA)}


\date{June 8, 2009.}


\begin{abstract}
We prove without any assumption on the ground field that higher Hochschild homology groups do not vanish for two large classes of algebras
whose global dimension is not finite.
\end{abstract}

\maketitle

\section{Introduction}

Let $k$ be a fixed field. All the algebras we consider are associative unital $k$-algebras. We will denote $\otimes = \otimes_k$.

It is well known that the homological properties of an algebra are related to the properties of its Hochschild (co)homology groups. 
For example, if a finite dimensional algebra over an algebraically closed field has finite global dimension, then all its higher 
Hochschild cohomology groups vanish.
In \cite{Hap}, D. Happel conjectured that the converse would be true. However, it has been shown in \cite{BGMS}
that the conjecture does not hold for algebras of type $A_q=k\langle x,y \rangle /(x^2, y^2, xy-qyx)$, where $q\in k$.

In \cite{Han}, Han proved that the total Hochschild homology of the algebras $A_q$ is infinite dimensional.
This fact led him to suggest the following conjecture:  

{\em {Conjecture}(Han): Let $A$ be a finite dimensional $k$-algebra. If the total Hochschild homology of $A$ is finite dimensional,
then $A$ has finite global dimension.}
 
In the same paper, Han provided a proof of this statement for monomial finite dimensional algebras.

Avramov and Vigu\'e's computations in \cite{AV} show that Han's conjecture holds in the commutative case not only for  finite dimensional algebras 
but for essentially finitely generated ones, see also \cite{Vi2}.

In \cite{BM}, Han's conjecture is proved for graded local algebras, Koszul algebras and graded cellular algebras,
provided the characteristic of the ground field is zero.
The proof relies on the properties of the graded Cartan matrix and the logarithm and strongly uses the hypothesis on the characteristic
of the field.

In \cite{BE}, the authors compute the Hochschild homology groups of quantum complete intersections, that is algebras of type
$A=k\langle x,y \rangle /(x^a, y^b, xy-qyx)$, where $q\in k^*$ is not a root of unity and $a,b\ge 2$ are fixed integers.
In particular they prove Han's conjecture for this class of finite dimensional algebras.

The main purpose of this paper is to prove that higher Hochschild homology groups do not vanish for two large classes of algebras
whose global dimension is not finite, without any assumption on the ground field.

In Theorem I, the algebras we consider are generalizations of quantum complete intersections and they are not assumed to be finite dimensional.

On the other hand, the algebras satisfying the hypotheses of Theorem II are, in some sense, opposite to quantum complete intersections, 
since we assume that they have two generators $x$ and $y$ such that $xy=yx=0$.

Now we state both main theorems. 

{\em {\bf Theorem I:} Let $A=k\langle x_1,\dots, x_n \rangle /(f_1,\dots, f_p)$ be a finitely generated $k$-algebra, 
such that $f_1$ belongs to $k[ x_1 ]$ and, for $i\ge 2$, $f_i$ belongs to the two-sided ideal $(x_2, \dots, x_n)$.
If $B= k[ x_1]/(f_1)$ is not smooth, then the Hochschild homology groups $HH_n(A)$ are not zero for an infinite 
increasing sequence of integers.}

For example Theorem I is valid if $f_1=x_1^2g_1$, with $g_1 \in k[x_1]$ and $f_2, \dots, f_p$ satisfying the hypothesis of the theorem.

\smallskip

{\em {\bf Theorem II:} Let $A= \bigoplus_{n\ge 0}A^n$ be a finite dimensional graded $k$-algebra with $A^0=k$ and such that  
$\overline{A}= \bigoplus_{n\ge 1}A^n$ is not zero.
Assume that there exist two generators $x$ and $y$ of the algebra $A$ verifying $xy=yx=0$.
Then the total Hochschild homology of $A$ is not finite dimensional.}

\begin{remark}
This theorem is valid for very large classes of graded local algebras since relations between the other generators 
play no role.
\end{remark} 

The proof of Theorem I follows without any computation from the well known result for commutative algebras.

The methods used in the proof of Theorem II rely on differential homological algebra.
In fact, we will work with the cobar construction on the graded coalgebra $\bigoplus_{n\ge 0}\Hom_k(A^n,k)$.  
We denote it $(\Omega^*A, d)$.
The Hochschild homology groups of the differential graded algebra $(\Omega^*A, d)$ are dual, as vector spaces, to the Hochschild
homology groups of the graded $k$-algebra $A$.
Since $(\Omega^*A, d)$ is a tensor algebra, a short complex is available to compute
its Hochschild homology.

\smallskip

The paper is organized as follows:
\begin{enumerate}
\item Introduction.
\item Proof of Theorem I.
\item Interpretation in terms of quivers.
\item Hochschild homology in the differential graded case.
\item Proof of Theorem II.
\end{enumerate}

\section{Proof of Theorem I}

Let $A$ be an associative unital $k$-algebra. The definition of the Hochschild homology groups, 
$HH_n(A)$, $n\ge 0$ is well known (see for example \cite{Lo}). We have 
\[ HH_n(A):=\Tor_n^{A^e}(A,A)= H_n(C_*(A), b) \]
where  $(C_*(A), b)$ is the Hochschild complex of $A$. 
Clearly, $HH_n(A)$ is a $k$-vector space for all $n\ge 0$.

In this section we assume that $A=k\langle x_1,\dots, x_n \rangle /(f_1,\dots, f_p)$
where $n,p \ge 1$, $f_1$, which we may suppose monic, belongs to $k[x_1]$ and, for $i\ge 2$, $f_i$ 
belongs to the two-sided ideal $(x_2, \dots, x_n)$.
Let us consider the $k$-algebra $B= k[x_1]/(f_1)$ and the maps
\[ \iota: B\to A \hbox{     with } \iota(x_1)=x_1, \]
\[ \pi: A\to B \hbox{     with } \pi(x_1)=x_1, \pi(x_i)=0, \hbox{ for } i\ge 2. \]

The following lemma is easy to prove.
\begin{lemma}
The maps $\iota$ and $\pi$ are morphisms of $k$-algebras and satisfy $\pi\circ\iota=id_B$.
\end{lemma}

Now, Theorem I is an immediate consequence of the following facts:
\begin{itemize}
\item the morphisms $\iota$ and $\pi$ induce by functoriality $k$-linear maps  
\[HH_*(\iota): HH_*(B) \to HH_*(A) \hbox{ and } HH_*(\pi): HH_*(A) \to HH_*(B)\] 
satisfying
$HH_*(\pi)\circ HH_*(\iota)=id_{HH_*(B)}$,
\item using a result of \cite{AV}, $HH_n(B)$ is non zero  for an infinite sequence of integers $n$.
\end{itemize}

Another proof can be given using the computations for  $HH_n(B)$ in \cite{BACH}:
if $f_1$ and $f_1'$ are not coprime, then $HH_n(B)\neq 0$ for all $n\in \NN$.

\begin{example}
If $f_1=x_1^a$, with $a\ge 2$, and $f_i \in (x_2, \dots, x_n)$, then Theorem I holds.
This covers the case of quantum complete intersections.  
\end{example}

An interesting question is to know if the algebras $A$ considered in Theorem I have infinite global dimension.
In the commutative case, it is well known that this is true.
Also, if $A=k\langle x_1,\dots, x_n \rangle /(f_1,\dots, f_p)$ is a finite dimensional $k$-vector space, Happel's result \cite{Hap} implies that
$\gldim(A)=\infty$, where $\gldim$ denotes the global dimension of the algebra. 

It follows from Serre's theorem in page 37 of \cite{Se} that if $B$ is not smooth, then its global dimension is not finite.
In the general case, we cannot ensure that if we have $k$-algebras $A$ and $B$ as above with $\gldim(B)=\infty$, then $\gldim(A)=\infty$.

However, we can use the algebra map $\iota:B\to A$  to obtain that the
global dimension of $A$ is not finite in some cases:
Suppose that $\iota$ endows $A$ with a structure of  flat $B$-module. In this situation, Corollary 4.4 of
\cite{Ba} says that $\gldim(A)=\infty$. This is the case, for example, of quantum complete intersections.

\section{Interpretation in terms of quivers}

Let $A$ be a finite dimensional basic $k$-algebra, then there exist a quiver $Q^A$ and an admissible ideal $I^A$ such  that $A$ is isomorphic to 
$kQ^A/I^A$.
In other words, if we denote by $Q_0^A=\{e_1, \dots, e_r \}$ the set of vertices of $Q^A$ and by $Q_1^A$ its set of 
arrows, then $kQ_0^A$ is an
algebra, $kQ_1^A$ is a $kQ_0^A$ two-sided ideal and  $A=T_{kQ_0^A}kQ_1^A/I^A$, where $I^A\subseteq(kQ_1^A)^2$.

Suppose that there exist $e_i\in kQ_0^A$ and $x\in e_i(kQ_1^A)e_i$. In fact, since $A$ is finite dimensional and $I^A$ is admissible, 
if such a loop $x$ exists then $x^n=0$ for some integer $n {\ge 2}$.

Let $B$ be the $k$-algebra $k[x]/\langle x^n \rangle$, then $B=T_{kQ_0^B}kQ_1^B/I^B$, where $Q_0^B=\{ e_i\}$, $Q_1^B=\{ x\}$ and 
$I^B=\langle x^n \rangle$.

We may  consider the morphisms of algebras of the previous section.
In this case the map $\iota$ is completely determined by its values on $e_i$ and $x$. 
It sends $e_i$ to $e_1 + \dots + e_r$ and $x$ to $x$. Clearly, it is well defined.

On the other hand, the morphism $\pi:A\to B$ is given as follows, $\pi(e_j)=\delta_{ij}e_i$, 
for $1\le j\le r$, and the restriction of $\pi$ to the arrows of $A$ is given by $\pi(y)=\delta_{yx}x$, where $\delta$ is the Kronecker 
delta.
If we assume that $I^A=\langle x^n, f_2 \dots, f_s \rangle$ is admissible and that $f_i$ belongs to the two-sided ideal
generated by $Q_1^A-\{ x\}$, then 
it is straightforward to check that $\pi$ is also well defined and $\pi\circ \iota=id_B$.

As a consequence of the results of Section 2, we see that the Hochschild homology dimension, denoted $\hhdim(B)$, is infinite 
and so the same holds for $A$.
Being both $k$-finite dimensional, their global dimensions cannot be finite.

It is interesting to note that analogous situations hold in several cases, for example, using results of \cite{Han}, each time
we have $char(k)=0$, $B$ monomial and $\hhdim(B)\neq 0$.

\section{Hochschild homology and cobar construction}

In this section we deal with finite dimensional algebras.

\subsection{Notation}

We use the methods of differential graded algebra of \cite{FHT1}. In particular an element of lower degree 
$i\in \ZZ$ is, by the classical convention, of upper degree $-i$. All the algebras considered from now on 
are unital, associative, with a differential of degree $-1$. 
We recall that if $V=\bigoplus_{i\in \ZZ}V_i$ is a graded $k$-vector space, then the suspended graded $k$-vector space $sV$ has 
homogeneous components
$(sV)_i=V_{i-1}$, for $i\in \ZZ$. The $k$-algebra $TV$ will denote the tensor algebra
on $V$. The degree of an element $v\in V$ is denoted $|v|$.

For any differential graded algebra $A$, let $A^{op}$ be the opposite graded algebra, and $A^e=A\otimes A^{op}$ be the enveloping algebra.
The categories of graded $A$-bimodules and of left (or right) differential graded $A^e$-modules are 
equivalent.

\subsection{Bar resolution and Hochschild homology}

Let $(A,d)$ be an augmented algebra and $\overline{A}=\Ker(\epsilon: A\to k)$.
The normalized bar 
resolution of $A$, denoted $B(A,A,A)$, is the differential graded $A^e$-module 
$(A\otimes T(s\overline{A})\otimes A, D_0+D_1)$, where $D_0$ is the differential induced by $d$ 
on the tensor product of complexes and $D_1$ is defined as follows (see for example
\cite{FTV}, 2.2.)
\begin{align*}
D_1(a\otimes sa_1 \otimes \dots \otimes sa_n\otimes b)=&
(-1)^{|a|}aa_1 \otimes sa_2 \otimes \dots \otimes sa_n \otimes b\\
&\pm \sum_{i=1}^{n-1}a \otimes sa_1\otimes \dots \otimes s(a_ia_{i+1}) \otimes \dots \otimes sa_n \otimes b\\
&\pm a \otimes sa_1 \otimes \dots \otimes sa_{n-1} \otimes a_nb.  
\end{align*}
The Hochschild homology of the differential graded algebra $(A,d)$ is, by definition, the graded vector space 
$\mathcal{HH}_*(A) = \Tor^{A^e}_*(A,A)$ in the differential sense of \cite{MacL}.

\begin{lemma}\cite{FHT1}
The canonical map $m:B(A,A,A) \to A$ defined by $0$ on $A \otimes T^{\ge 1}(s\overline{A})\otimes A$, 
and by  multiplication on $A\otimes A$ is a semifree resolution of $A$ as an $A^e$-module.
\end{lemma}

Consequently we have, 
\[ \mathcal{HH}_*(A, d) = H_*(\mathcal{C}_*(A), \delta ) \]
with 
\[ \mathcal{C}_*(A)= A\otimes_{A^e}B(A,A,A) = A\otimes T(s\overline{A}), \]
and $\delta =\delta _0+\delta _1$, where $\delta _0$ and $\delta _1$ are obtained by tensorization.

Explicitly, 
\begin{align*}
\delta _1(a\otimes sa_1 \otimes \dots \otimes sa_n)=& 
(-1)^{|a|}aa_1 \otimes sa_2 \otimes \dots \otimes sa_n\\
& +\sum_{i=1}^{n-1} (-1)^{\epsilon_i}a \otimes sa_1\otimes \dots \otimes s(a_ia_{i+1}) \otimes \dots \otimes sa_n \\
& +  (-1)^{\epsilon_n} a_na \otimes sa_1 \otimes \dots \otimes sa_{n-1} , 
\end{align*}
where the $\epsilon_i$'s are integers depending on the degrees of the elements $a_i$; if all these
degrees are even, then $\epsilon_i=i$.

In the rest of this paper we consider only differential graded algebras $(A,d)$ satisfying 
either condition (a) or condition (b) below.

\begin{itemize}
 \item[(a)] $A_n=0$ for $n<0$ and $A_0=k$, so that $\mathcal{C}_n(A)=0$ for $n<0$;
 \item[(b)] $A_n=0$ for $n>0$, $A_0=k$, $A_{-1}=0$, so that $\mathcal{C}_n(A)=0$ for $n>0$.
\end{itemize}

In both cases, we have $\mathcal{C}_0(A)=k$.

\subsection{Cobar construction and duality construction in Hochschild homology}

We next recall the definition of the cobar construction described in Section 19 of \cite{FHT2}.
Let $(C,d_{C})$ be a coaugmented differential graded coalgebra with comultiplication
$\Delta$, and $\overline{C}=\Ker(\epsilon:C\to k)$.
We denote $(\Omega C, d)$ the augmented differential graded algebra defined as follows:
\begin{itemize}
 \item $\Omega C= T(s^{-1}\overline{C})$, as augmented graded algebra,
 \item $d=d_0+d_1$, with $d_0(s^{-1}c)= - s^{-1}(d_{C}(c))$, if $c \in \overline{C}$, and $d_1$
is defined from $\Delta$.
\end{itemize}
Suppose now that $(A,d_A)$ is a finite dimensional differential graded algebra, then the graded 
dual $A^{\vee}=Hom_k(A,k)$ is a differential graded coalgebra with differential $d_A^{\vee}$, the transpose of $d_A$.

\begin{definition} $(\Omega^*A, d):=(\Omega (A^{\vee}),d)$, where $d$ is defined from $d_A^{\vee}$ and the comultiplication of 
$A^{\vee}$ as above.
\end{definition}

We have $\Omega^*A=T(V)$ with $V=Hom_k(s\overline{A},k)$.
If $(A,d_A)$ satisfies condition (b) above, then
\[ V=\bigoplus_{n\ge 1}V_n, \hbox{ with } V_n=Hom_k(A_{-n-1},k) \]
and then $(\Omega^*A, d)$ satisfies condition (a).
Similarly, if $(A,d_A)$ satisfies condition (a), then $(\Omega^*A, d)$ satisfies condition (b).

The first ingredient used to prove  Theorem II is the following duality property.

\begin{theorem}\label{HV-So}\cite{HV}, \cite{So}:
Let $(A,d_A)$ be a finite dimensional algebra satisfying either condition (a) or (b) above,
then for all $n\in \mathbb{Z}$ we have:
\[ \Hom_k(\mathcal{HH}_{-n}(A),k)=\mathcal{HH}_n(\Omega^*A). \]
\end{theorem}

Consequently, the computation of the graded vector space 
$\mathcal{HH}_n(A)$ can be replaced by the computation of the Hochschild homology of a quasifree differential graded
algebra $(T(V),d)$.

\subsection{A short complex for the computation of the Hochschild homology}\label{delta}

Now, we want to compute the Hochschild homology of $(T(V),d)$, with 
$V=\bigoplus_{n\ge 1}V_n$.

We recall here the main results of \cite{Vi1}. Put $(T(V),d)=(B,d)$ and
let $P=(B \otimes B) \oplus (B \otimes (sV)\otimes B)$, we define a differential $D$ on $P$,
which is the tensor product of the differentials on $B \otimes B$, and 
\[D(a\otimes sv \otimes b)= da \otimes sv \otimes b \pm (av\otimes b - a\otimes vb) +S(a \otimes sv \otimes b), \]
where $S(a \otimes sv \otimes b) \in B\otimes sV \otimes B,$ for $a\in B, b\in B$ and $v\in V$.
\begin{prop}(Thm. 1.4 in \cite{Vi1})
The canonical map $m:(P,D) \to B$ defined as $0$ on $B\otimes sV \otimes B$
and as multiplication on $B\otimes B$ is a semifree resolution of $B$ as
$B^e$-module.
\end{prop}

As  a consequence,
\[ \mathcal{HH}_*(T(V),d)=H_*(B\otimes_{B^e}P, \delta), \]
with differential $\delta=d\otimes_{B^e}D$ that will be precised in the next section. We have: 
\begin{itemize}
 \item $\delta_{|T(V)}=d$,
 \item $\delta(a\otimes sv)= da\otimes sv + (-1)^{|a|}(av-(-1)^{|v|\times|a|}va) - \sigma(a\otimes dv)$,
where $\sigma(a\otimes dv)$ belongs to $T(V)\otimes sV$, for $a\in T(V),v\in V$.
\end{itemize}

Put $Q_*:=B\otimes_{B^e}P=T(V)\oplus (T(V)\otimes sV)$.

\begin{theorem}\label{Vi1}(Thm. 1.5 of \cite{Vi1})
With the above notations, \[ \mathcal{HH}_*(T(V),d)=H_*(Q_*,\delta). \]
\end{theorem}

In the following section we will use the complex $(Q_*,\delta)$ to compute the 
Hochschild homology of a finite dimensional graded algebra
$A=\bigoplus_{n\ge 0}A^n$, with $A^0=k$.
In this case, the graded vector space $V$ is also finite dimensional, and the
differential $\delta$ has good properties.

\section{Proof of Theorem II}

We work with a finite dimensional graded algebra with $A^0=k$.
We may assume without loss of generality that $A$ is graded in even degrees, 
$A=k\oplus \left(\bigoplus_{n\ge 2}A^n\right)$, and $\overline{A}=\bigoplus_{n\ge 2}A^n$ 
is non zero.

\subsection{Relations between $HH_*(A)$ and $\mathcal{HH}_*(A,0)$}

Using the conventions recalled at the beginning of the previous section, we consider $A$
as a differential graded algebra with differential $0$ and $A_{-n}=A^n$.

Since $A$ is graded, the ordinary Hochschild homology $HH_*(A)$ defined in Section 2 
is graded, and there is a decomposition 
\[ HH_*(A)= \bigoplus_{p,q\ge 0}HH_p(A)^q . \]

Since $A$ is finite dimensional, $HH_p(A)$ is finite dimensional for all $p$.

\begin{lemma}
Let $A$ be an algebra as above. Then,
 \begin{enumerate}
 \item $\mathcal{HH}_*(A,0)= \bigoplus_{n\ge 0}\mathcal{HH}_{-n}(A)$ and 
$\mathcal{HH}_{-n}(A)= \bigoplus_{p}HH_p(A)^{p+n}$.
 \item $HH_p(A)^{p+n}=0$ if $p>n$ or $p<\frac{n-N}{N-1}$, where $N=sup\{n|A^n\neq 0\}$.
\end{enumerate}
\end{lemma}

\begin{coro}
 If there exists an increasing sequence of integers $n_i$ such that 
$\mathcal{HH}_{-n_i}(A)\neq 0$, then $HH_*(A)$ is not finite dimensional.
\end{coro}

The strategy now is to focus our attention on
$\mathcal{HH}_{*}(\Omega^*A)$, using Theorem \ref{HV-So}. But
Theorem \ref{Vi1} allows us to use the short complex $(Q_*,\delta)$ to compute 
$\mathcal{HH}_{*}(\Omega^*A)$, so we will work with this last one.

\subsection{Description of $(Q_*,\delta)$}

Let $A=k\oplus \left(\bigoplus_{n\ge 2}A^n\right)$ be a finite dimensional graded algebra.
We fix a homogeneous linear basis $(a_i)_{i\in I}$ for $\overline{A}=\bigoplus_{n\ge 2}A^n$.
This choice determines the structure constants $\alpha^i_{jk}$ by the equalities
$a_ja_k=\sum \alpha^i_{jk}a_i$.

In this situation,  $(\overline{A})^{\vee}=\Hom_k(\overline{A},k)$ is endowed with the dual basis
$(b_i)_{i \in I}$ satisfying $\langle b_i,a_j\rangle= \delta_{ij}$.
Notice that $A^{\vee}$ is a graded coalgebra with comultiplication $\Delta$,
and  $\Delta b_i = \sum_{j.k}\beta^{jk}_i b_j\otimes b_k$, where 
$\alpha^i_{jk}=(-1)^{|a_j|\times|a_k|}\beta^{jk}_i$.

We have already defined $(\Omega^*A,d)=(\Omega(A^{\vee}),d)=(T(V),d)$. Now, put $v_i=s^{-1}b_i$, 
then $|v_i|=n-1$ if $a_i \in A^n$.
We check that 
\[ dv_i = \sum_{j,k}(-1)^{|a_j|+|a_j|\times|a_k|}\alpha^i_{jk}v_j\otimes v_k . \]
So $(\Omega^*A,d)=(T(V),d)$ is a tensor algebra with a quadratic differential.

Furthermore, we have assumed without loss of generality that $A$ is graded in even degrees,
so that $V$ is graded only in odd degrees.
In this case, we give an explicit formula for the differential $\delta$
on $Q_*$ (cf. Subsection \ref{delta} ). 

Put $\overline{V}=sV$, then $Q_*=T(V)\oplus T(V)\otimes \overline{V}$.
Let $v$ be an element in $V$, and $dv=\sum_{j,k}\lambda_{jk}v_j\otimes v_k$, 
with $\lambda_{jk} \in k$. Let $a$ be an element in $T(V)$.

We have:
\[ \delta(a\otimes \overline{v})=da\otimes \overline{v}+(-1)^{|a|}(av-(-1)^{|a|}va)-\sigma(a\otimes dv),  \]
where
\[ \sigma(a\otimes dv)=-(-1)^{|a|}\sum_{j,k} \lambda_{jk}av_j\otimes \overline{v}_k
+ \sum_{j,k} \lambda_{jk}v_ka\otimes \overline{v}_j. \]

\subsection{A nice homogeneous basis $(a_i)$ for $\overline{A}$}

Since $A=k\oplus \overline{A}$, the projection $\overline{A}\to \overline{A}/\overline{A}^2=U$
has a section $\rho$ that extends to a morphism of algebras $T(U)\to A$ whose kernel is 
contained in $T^{\ge 2}(U)$.
This implies that $(x_i)_{1\le i \le p}$ are generators of the algebra
$A$ if and only if their images in $\overline{A}/\overline{A}^2$ form a basis of this vector 
space.

As vector spaces, $\overline{A}= \overline{A}/\overline{A}^2\oplus \overline{A}^2$,
and we will consider a homogeneous basis of $\overline{A}/\overline{A}^2$ and a 
basis of $\overline{A}^2$.
If $a_i\in \overline{A}/\overline{A}^2$, then the corresponding $v_i$ in
$(\Omega^*A,d)$ satisfies $dv_i=0$.

\smallskip

We will now prove the following result.

\begin{theorem}
 Let $A=\bigoplus_{n\ge 0}A^n$ be a finite dimensional graded $k$-algebra with $A^0=k$,
such that $\overline{A}=\bigoplus_{n\ge 1}A^n$ is not zero.
Assume that there exist two generators $x$ and $y$ of the algebra 
$A$ satisfying $xy=yx=0$, then $H_{n_i}(Q_*,\delta)\neq 0$ for a strictly increasing sequence
of integers $(n_i)$.
\end{theorem}
\begin{proof} 
We can associate to $x$ and $y$ two elements $a_1$ and $a_2$, linearly independent
in $\overline{A}$. We denote by $v_1$ and $v_2$ the corresponding elements in a dual basis 
of $V$.
If $(a_1,\dots, a_n)$ is a linear basis of $\overline{A}$ and $(v_1,\dots, v_n)$ is the
corresponding basis of $V$, then we have $dv_1=0$, $dv_2=0$ and for $i\ge 3$,
\[ dv_i= \sum_{j,k}\alpha^i_{jk}v_j\otimes v_k. \]
The fact that $xy=yx=0$ implies that, for $i\ge 3$, $\alpha^i_{12}=\alpha^i_{21}=0$.

For $n\ge 1$, consider:
\[ X_n=v_1\otimes v_2 \otimes v_1 \otimes v_2 \otimes \dots \otimes v_1 \otimes \overline{v}_2
- v_2\otimes v_1 \otimes v_2 \otimes v_1 \otimes \dots \otimes v_2 \otimes \overline{v}_1 \in 
V^{\otimes (2n-1)}\otimes \overline{V}. \]
It is easy to see that 
$|X_n|= n(|v_1|+|v_2|)+1$ and that $\delta X_n=0$.

If $X_n$ was a boundary, it should exist $Y, b_i\in T(V)$ such that 
$X_n=\delta(Y+\sum_i b_i\otimes \overline{v}_i)$ and
\[ X_n=dY+ \sum_idb_i\otimes \overline{v}_i + \sum_i(b_iv_i-v_ib_i)+ 
\sum_i\alpha^i_{jk}b_i v_j\otimes \overline{v}_k - \sum_i\alpha^i_{jk}v_kb_i\otimes \overline{v}_j. \]
Such elements cannot exist since, for all $i$, 
\[ dv_i= \sum_{j,k}\alpha^i_{jk}v_j \otimes v_k \hbox{ with } \alpha^i_{12}=\alpha^i_{21}=0. \]
\end{proof}

\begin{example}
Let $A=k\langle x,y,z \rangle /(xy, yx, x^2-y^2, x^2-z^2, xz-qzx, yz-qzy)$ where
$q\in k$, $q^2\neq 1$ and $-1$ is not a square in $k$.
This example is not covered by Theorem I.
\end{example}

\bibliographystyle{amsplain}

\begin{thebibliography}{99}

\bibitem{AV} Avramov, L.; Vigu\'e-Poirrier, M. \textit{Hochschild homology criteria for smoothness}.  Internat. Math. Res. Notices  {\textbf 1} (1992),  17--25.

\bibitem{Ba} Bavula, V. V.  \textit{Tensor homological minimal algebras, global dimension of the tensor product of algebras and of generalized 
Weyl algebras}, Bull. Sci. Math., {\textbf 120} (1996), no. 3, 293--335.

\bibitem{BE} Bergh, P. A.; Erdmann, K. \textit{ Homology and cohomology of quantum complete intersections}.  
Algebra Number Theory  {\textbf 2}  (2008),  no. 5, 501--522

\bibitem{BM} Bergh, P. A.; Madsen, D. \textit{Hochschild homology and global dimension}. Bull. London Math. Soc., to appear.
\texttt{arXiv:0803.3550}

\bibitem{BGMS} Buchweitz, R.; Green, E.; Madsen, D.; Solberg, O. \textit{Hochschild cohomology without finite global dimension}. Math. Res. Let. 
{\textbf 12} (2005), 805--816.

\bibitem{BACH} Buenos Aires Cyclic Homology Group. \textit{Cyclic homology of algebras with one generator}. J. A. Guccione, J. J. Guccione, M. J. Redondo, A. Solotar and O. Villamayor participated in this research.  $K$-Theory  {\textbf 5}  (1991),  51--69. 

\bibitem{FHT1} F\'elix, Y.; Halperin, S.; Thomas, J.-C. \textit{Differential graded algebras in topology}.  Handbook of algebraic topology,  829--865, North-Holland, Amsterdam, 1995.

\bibitem{FHT2} F\'elix, Y.; Halperin, S.; Thomas, J.-C. \textit{Rational homotopy theory}. 
Graduate Texts in Mathematics, vol. 205, Springer-Verlag, New York, 2001. 

\bibitem{FTV} F\'elix, Y.; Thomas, J.-C.; Vigu\'e-Poirrier, M. 
\textit{The Hochschild cohomology of a closed manifold}.  Publ. Math. Inst. Hautes \'Etudes Sci. 
{\textbf 99}  (2004), 235--252. 

\bibitem{HV} Halperin, S.; Vigu\'e-Poirrier, M. \textit{The homology of a free loop space}.  
Pacific J. Math.  {\textbf 147}  (1991),  no. 2, 311--324.

\bibitem{Han} Han, Y. \textit{Hochschild (co)homology dimension}.  J. London Math. Soc. (2)  {\textbf 73}  (2006),  no. 3, 657--668. 

\bibitem{Hap} Happel, D. \textit{Hochschild cohomology of finite-dimensional algebras}.  
S\'eminaire d'Alg\`ebre Paul Dubreil et Marie-Paul Malliavin, 39\`eme Ann\'ee (Paris, 1987/1988), 108--126, 
Lecture Notes in Math., {\textbf 1404}, Springer, Berlin, 1989.

\bibitem{Lo} Loday, J.-L. \textit{Cyclic homology}. Appendix E by M. Ronco. Second edition. 
Chapter 13 by the author in collaboration with Teimuraz Pirashvili. Grundlehren der Mathematischen Wissenschaften, vol. 301, Springer-Verlag, 
Berlin, 1998.

\bibitem{MacL} MacLane, S. \textit{Homology}. Reprint of the first edition. 
Die Grundlehren der mathematischen Wissenschaften, vol. 114. Springer-Verlag, Berlin-New York, 1967. 

\bibitem{Se} Serre, J.-P. \textit{Alg\`ebre locale. Multiplicit\'es}. Lecture Notes in Math. {\textbf 11} (1965). Springer-Verlag, Berlin.

\bibitem{So} Solotar, A. \textit{Cyclic homology of a free loop space}.  
Comm. Algebra  {\textbf 21}  (1993),  no. 2, 575--582.

\bibitem{Vi1} Vigu\'e-Poirrier, M. \textit{Homologie de Hochschild et homologie cyclique des 
alg\`ebres diff\'erentielles gradu\'ees}. 
International Conference on Homotopy Theory (Marseille-Luminy, 1988).  Ast\'erisque, vol. 191, Soc. Math. France,  1990, pp. 255--267.

\bibitem{Vi2} Vigu\'e-Poirrier, M. \textit{Crit\`eres de nullit\'e pour l'homologie 
des alg\`ebres gradu\'ees}. C. R. Acad. Sci. Paris S\'er. I Math.  {\textbf 317}  (1993),  no. 7, 647--649.

\end{thebibliography}

\end{document}